\documentclass{amsart}

\usepackage[english]{babel}

\usepackage{amssymb}
\usepackage{mathrsfs}

\usepackage{hyperref}

\usepackage{comment}
\usepackage{graphicx,color}
\usepackage[all]{xy}
\usepackage{tikz}
\usepackage{tikz-cd}

\usepackage[initials]{amsrefs}

\newtheorem{theorem}{Theorem}[section]
\newtheorem{lemma}[theorem]{Lemma}
\newtheorem{proposition}[theorem]{Proposition}
\newtheorem{cor}[theorem]{Corollary}

\theoremstyle{definition}
\newtheorem{definition}[theorem]{Definition}
\newtheorem{remark}[theorem]{Remark}
\newtheorem{example}[theorem]{Example}

\newcommand{\into}{\hookrightarrow}
\newcommand{\onto}{\twoheadrightarrow}
\renewcommand{\setminus}{\smallsetminus}
\newcommand{\phiv}{\varphi}

\newcommand{\inv}{^{-1}}
\DeclareMathOperator{\coker}{coker}

\newcommand{\Ext}{\mathrm{Ext}}
\newcommand{\Hom}{\mathrm{Hom}}
\newcommand{\cExt}{\mathcal{E}xt}

\DeclareMathOperator{\Spec}{Spec}
\newcommand{\dP}{\mathrm{dP}}

\def\HilbS#1{\mathrm{Hilb}( #1 , -K_{#1})}
\def\pow#1{[ \! [ #1 ] \! ] }
\def\art{\mathrm{(Art)}}
\def\comp{\mathrm{(Comp)}}

\def\Def{\mathit{Def}}
\def\Deflt{\mathit{Def}^\mathrm{lt}}
\def\Defloc{\mathit{Def}^\mathrm{loc}}

\def\cone#1{\mathrm{cone} \left\{ #1 \right\}}
\def\conv#1{\mathrm{conv} \left\{ #1  \right\} } 

\def\vectortwo#1#2{\begin{pmatrix}
#1 \\ #2
\end{pmatrix}}

\def\vector#1#2#3{\begin{pmatrix}
#1 \\ #2 \\ #3
\end{pmatrix}}

\def\vectorfour#1#2#3#4{\begin{pmatrix}
#1 \\ #2 \\ #3 \\ #4
\end{pmatrix}}

\def\listpoly#1{\mathcal{S}_{\text{#1}}}

\newcommand\cE{\mathcal{E}}
\newcommand\cO{\mathcal{O}}

\newcommand\cX{\mathcal{X}}

\renewcommand\AA{\mathbb{A}}
\newcommand\CC{\mathbb{C}}
\newcommand\GG{\mathbb{G}}

\newcommand\PP{\mathbb{P}}
\newcommand\QQ{\mathbb{Q}}
\newcommand\RR{\mathbb{R}}
\newcommand\ZZ{\mathbb{Z}}

\newcommand\rH{\mathrm{H}}

\newcommand\rT{\mathrm{T}}
\newcommand\rV{\mathrm{V}}

\newcommand\rmm{\mathrm{m}}

\newcommand{\frakm}{\mathfrak{m}}
\newcommand{\frakp}{\mathfrak{p}}

\title{On deformations of toric Fano varieties}

\author{Andrea Petracci}

\address{Institut f\"ur Mathematik,
Freie Universit\"at Berlin,
Arnimallee 3, 
14195 Berlin, 
Germany}

\email{andrea.petracci@fu-berlin.de}

\begin{document}

\begin{abstract}
In this note we collect some results on the deformation theory of toric Fano varieties.
\end{abstract}

\maketitle

\section{Introduction}

A \emph{Fano variety} is a normal projective variety $X$ over $\CC$ such that its anticanonical divisor $-K_X$ is $\QQ$-Cartier and ample.  Fano varieties constitute the basic building blocks of algebraic varieties, according to the Minimal Model Program.
The geometry of Fano varieties is a well studied area. In particular, moduli (and consequently deformations) of Fano varieties constitute a very interesting and important topic in algebraic geometry, e.g.\ \cite{kmm_boundedness_fano, birkar, xu_moduli}.

Here we will concentrate on deformations and smoothings of \emph{toric} Fano varieties. These varieties occupy a prominent role in Mirror Symmetry, a large part of which is based on the phenomenon of toric degeneration as in \cite{batyrev_ciocanfontanine, gs_annals, mirror_symmetry_fano, toric_degenerations_fano}.

Toric Fano varieties correspond to certain polytopes which are called \emph{Fano polytopes}.
The goal of this note is to present some combinatorial criteria on Fano polytopes which can detect whether the corresponding toric Fano variety is smoothable, i.e.\ can be deformed to a smooth (Fano) variety.

Special attention is given to toric Fano threefolds with Gorenstein singularities.
These varieties correspond to the 4319 reflexive polytopes of dimension 3, which were classified by Kreuzer and Skarke \cite{kreuzer_skarke_reflexive_3topes}.
In this case, thanks to the use of the software MAGMA, we were able to produce a lot of examples for the combinatorial criteria discussed in this note.

\subsection*{Outline}

In \S\ref{sec:infinitesimal_deformations} the very classical theory of infinitesimal deformations of algebraic varieties is recalled. In \S\ref{sec:smoothings} we survey some properties of smoothings of algebraic varieties. In \S\ref{sec:invariants} two well-studied deformation invariants for Fano varieties are introduced.

In \S\ref{sec:toric_singularities} we recall some results on the deformation theory of affine toric varieties.
We provide an example in \S\ref{sec:affine_cone_dP7}.

The core of this note is \S\ref{sec:deformations_toric_fano}. We recall the definition of Fano polytopes in \S\ref{sec:Fano_polytopes}.
In \S\ref{sec:toric_rigid_singularity_then_non_smoothable} we present a couple of sufficient conditions that ensure that a toric Fano variety is non-smoothable.
The rigidity of toric Fano varieties is examined in \S\ref{sec:rigidity}.
In \S\ref{sec:toric_dP} and \S\ref{sec:toric_fano_3folds_with_isolated_singul} we study the smoothability of toric Fano surfaces and toric Fano threefolds with isolated singularities; an example is presented in \S\ref{sec:projective_cone_dP7}.
In \S\ref{sec:almost_flat_triangles} we present another sufficient condition that ensures that a toric Fano threefold is non-smoothable.
In \S\ref{sec:other} we include more results on deformations of toric Fano varieties.

In \S\ref{sec:lists} we write down the lists of the reflexive polytopes of dimension 3 which satisfy the several combinatorial conditions considered in \S\ref{sec:deformations_toric_fano}.

\subsection*{Notation and conventions}

We work over $\CC$, but everything will hold over a field of characteristic zero with appropriate modifications.

In \S\ref{sec:deformations_toric_singularities} and \S\ref{sec:deformations_toric_fano} we assume that the reader is familiar with the basic notions of toric geometry, which can be found in \cite{cox_toric_varieties, fulton_toric_varieties}.
All toric varieties considered here are normal.
A lattice is a finitely generated abelian group. The letters $N, \overline{N}, \tilde{N}$ stand for lattices and $M, \overline{M}, \tilde{M}$ for their duals, e.g.\ $M = \Hom_\ZZ  (N, \ZZ)$; the duality pairing $M \times N \to \ZZ$ and its extension $M_\RR \times N_\RR \to \RR$ are denoted by $\langle \cdot, \cdot \rangle$.

In a real vector space of finite dimension a polytope is the convex hull of finitely many points, or equivalently a compact subset which is the intersection of finitely many closed halfspaces. We refer the reader to the book  \cite{ziegler} for the geometry of polytopes.

\subsection*{Acknowledgements} 
I am indebted to Alessio Corti, Paul Hacking, Alexander Kasprzyk, and Thomas Prince for many fruitful conversations about the topics of this survey.
I am very grateful to Tom Coates and Alexander Kasprzyk for having provided me with access to their MAGMA machine and to their database of Fano polytopes.
Finally, I would like to thank the anonymous referee for having done a thorough job.

\section{Deformations}
\label{sec:deformations}

\subsection{Infinitesimal deformations}
\label{sec:infinitesimal_deformations}
Let $\comp$ be the category of noetherian complete local $\CC$-algebras with residue field $\CC$.
For every $R \in \comp$ we denote by $\frakm_R$ the maximal ideal of $R$.
Let $\art$ be the subcategory of $\comp$ whose objects are artinian, i.e.\ local finite $\CC$-algebras. A  functor of Artin rings is a functor $F$ from the category $\art$ to the category of sets such that $F(\CC)$ is the set with one element.
We will only consider functor of Artin rings which satisfy some additional properties: Schlessinger's axioms (H1) and (H2) \cite{schlessinger_functors} and Fantechi--Manetti condition (L) \cite[(2.9)]{fantechi_manetti}.
We will not specify these conditions here, but we refer the reader to \cite[\S2]{fantechi_manetti} for a quick introduction.
Precise formulations and additional details about the notions we introduce below can be found in any reference about deformation theory, e.g.\ \cite{fantechi_fga_deformation, hartshorne_deformation_theory, sernesi_deformation, talpo_vistoli_deformation, manetti_seattle, artin_lectures_deformation, schlessinger_functors, vistoli_lci}.

A natural transformation (or briefly map) of functors $\phi \colon F \to G$ is called \emph{smooth} if the lifting property in Grothendieck's definition of formally smooth morphisms holds, i.e.\ for every local surjection $A' \onto A$ in $\art$ the natural map $F(A') \to F(A) \times_{G(A)} G(A')$ is surjective;
in particular, if $\phi$ is smooth then $\phi(A) \colon F(A) \to G(A)$ is surjective for all $A \in \art$. A functor $F$  is called smooth if the map from $F$ to the trivial functor is smooth.

For a functor $F$, the set $F(\CC[t]/(t^2))$ has a natural structure of a $\CC$-vector space, denoted by $\rT F$ and called the \emph{tangent space} of $F$. One can prove that $F$ is the trivial functor if and only if $\rT F = 0$. If $\phi \colon F \to G$ is a map, then the function $\phi(\CC[t]/(t^2))$ is linear and denoted by $\rT \phi \colon \rT F \to \rT G$.

If $R \in \comp$ one can consider the functor $h_R = \Hom(\cdot,R)$ prorepresented by $R$.
A map $h_R \to F$ is equivalent to a pro-object of $F$ on $R = \varprojlim R / \frakm_R^{n+1}$, i.e.\ an element of the set $\varprojlim F(R/ \frakm_R^{n+1})$. 
A \emph{hull} for a functor $F$ is a ring $R \in \comp$ together with a smooth morphism $\phi \colon h_R \to F$ such that $\rT \phi$ is bijective. A hull exists if and only if $\rT F$ has finite dimension. If a hull exists, it is unique. Provided that $\rT F$ has finite dimension $r$, then $F$ is smooth if and only if the hull of $F$ is isomorphic to $\CC \pow{t_1, \dots, t_r}$.

For a functor $F$, consider the set $\cE$ made up of pairs $(\pi, \xi)$, where $\pi \colon A' \to A$ is a surjection in $\art$ such that $\frakm_{A'} \cdot (\ker \pi) = 0$ and $\xi \in F(A)$. A $\CC$-vector space $V$ is called an \emph{obstruction space} for $F$ if there exists a function $\omega \colon \cE \to \coprod_{(\pi,\xi) \in \cE} \ker \pi \otimes_\CC V$ such that the two following conditions are satisfied:
\begin{itemize}
\item for every $(\pi,\xi) \in \cE$, $\omega(\pi,\xi) \in \ker \pi \otimes_\CC V$;
\item for every $(\pi,\xi) \in \cE$, we have that $\omega(\pi,\xi) = 0$ if and only if there exists $\xi' \in F(A')$ which maps to $\xi$.
\end{itemize}
There are infinitely many obstruction spaces for a functor $F$ because any vector space containing an obstruction space is an obstruction space. A functor $F$ is smooth if and only if $0$ is an obstruction space for $F$; in this case we also say that $F$ is unobstructed. There is a notion of compatible obstruction spaces for a map $\phi \colon F \to G$: this will be a linear map $\mathrm{o} \phi$ from an obstruction space of $F$ to an obstruction space of $G$ with some compatibility properties with respect to $\phi$.

The following is an important smoothness criterion. Assume that $\phi \colon F \to G$ is a map with compatible obstruction map $\mathrm{o} \phi$ from an obstruction space of $F$ to an obstruction space of $G$. If $\rT \phi$ is surjective and $\mathrm{o} \phi$ is injective, then $\phi$ is smooth.

\bigskip

Let $X$ be a scheme of finite type over $\CC$. We denote by $\Def_X$ the functor of (infinitesimal) deformations of $X$. If $R \in \comp$, a pro-object of $\Def_X$ on $R$ is called a \emph{formal deformation} of $X$ over $R$. If $R$ is a hull for $\Def_X$, then the corresponding formal deformation of $X$ over $R$ is called the \emph{miniversal deformation} of $X$.
We say that $X$ is \emph{rigid} if all deformations of $X$ are trivial.
If $X$ is reduced, then the tangent space of $\Def_X$ is $\Ext^1(\Omega_X, \cO_X)$;
in this case $X$ is rigid if and only if $\Ext^1(\Omega_X, \cO_X) = 0$.
If $X$ is either normal or reduced and local complete intersection (l.c.i.\ for short), then $\Ext^2(\Omega_X, \cO_X)$ is an obstruction space for $\Def_X$.
If $X$ is smooth, then $\rH^i(X,T_X) = \Ext^i(\Omega_X, \cO_X)$ for all $i \geq 0$. In particular, if $X$ is smooth and affine then it is rigid.

\begin{proposition}
If $X$ is a smooth Fano variety, then $\rH^i (X,T_X) = 0$ for each $i \geq 2$. In particular, the infinitesimal deformations of $X$ are unobstructed, i.e.\ $\Def_X$ is smooth.
\end{proposition}

\begin{proof}
Let $n$ be the dimension of $X$. Since the anticanonical line bundle $\omega_X^\vee$ is ample, by Kodaira--Nakano vanishing we have $\rH^i(X, \Omega_X^{n-1} \otimes \omega_X^\vee) = 0$ whenever $i + n-1 > n$, i.e.\ $i \geq 2$. We conclude because the tangent sheaf $T_X$ is isomorphic to $\Omega_X^{n-1} \otimes \omega_X^\vee$. 
\end{proof}

Let $X$ be a scheme of finite type over $\CC$ and let $\Deflt_X$ be the subfunctor of $\Def_X$ made up of the locally trivial deformations of $X$. The tangent space of $\Deflt_X$ is $\rH^1(X, T_X)$ and $\rH^2(X,T_X)$ is an obstruction space for $\Deflt_X$.

\begin{proposition} \label{prop:locally_trivial_deformations}
Let $X$ be a reduced scheme of finite type over $\CC$ such that $X$ is either l.c.i.\ or normal.
If $\rH^0(X, \cExt^1(\Omega_X, \cO_X)) = 0$, then all deformations of $X$ are locally trivial, i.e.\ $\Deflt_X = \Def_X$.
\end{proposition}

\begin{proof}
The local-to-global spectral sequence for Ext gives the following exact sequence.
\begin{equation*}
0 \to \rH^1(T_X) \to \Ext^1(\Omega_X, \cO_X) \to \rH^0(\cExt^1(\Omega_X, \cO_X)) \to \rH^2(T_X) \to \Ext^2(\Omega_X, \cO_X)
\end{equation*}
The vanishing of $\rH^0(\cExt^1(\Omega_X, \cO_X))$ implies that the inclusion $\phi \colon \Deflt_X \into \Def_X$ induces an isomorphism on tangent spaces and an injection on obstruction spaces. Therefore $\phi$ is smooth, and consequently surjective.
\end{proof}

In particular, all deformations of a smooth scheme are locally trivial.

\bigskip

Let $X$ be a reduced scheme of finite type over $\CC$ with isolated singularities. For each singular point $x \in X$, let $U_x$ be an affine open neighbourhood of $x$ such that $U_x \setminus \{ x \}$ is smooth. Then define
\[
\Defloc_X := \prod_{x \in \mathrm{Sing}(X)} \Def_{U_x}.
\]
The tangent space of $\Defloc_X$ is $\rH^0(X, \cExt^1(\Omega_X, \cO_X))$. If $X$ is either l.c.i.\ or normal, then $\rH^0(X, \cExt^2(\Omega_X, \cO_X))$ is an obstruction space for $\Defloc_X$.
There is an obvious map $\Def_X \to \Defloc_X$ which restricts a deformation of $X$ to a deformation of $U_x$ for each $x$.

\begin{proposition} \label{prop:no_local_to_global_obstructions}
Let $X$ be a reduced scheme of finite type over $\CC$ with isolated singularities. Assume that $X$ is either l.c.i.\ or normal.
If $\rH^2(X,T_X) = 0$ then there are no local-to-global obstructions for the infinitesimal deformations of $X$, i.e.\ the map $\Def_X \to \Defloc_X$ is smooth.
\end{proposition}

\begin{proof}
We consider the local-to-global spectral sequence for $\Ext^\bullet(\Omega_X, \cO_X)$.
The second page is given by $E_2^{p,q} = \rH^p(\cExt^q(\Omega_X, \cO_X))$.
Since $X$ has isolated singularities, the sheaves $\cExt^q(\Omega_X, \cO_X)$ are supported on isolated points for $q \geq 1$; in particular they do not have higher cohomology.
This means that $E_2^{p,q}$ is supported on the lines $p=0$ and $q=0$. Therefore, in $E_2$ the only non-zero differential is
\[
d_2 \colon \rH^0(\cExt^1(\Omega_X, \cO_X)) \longrightarrow \rH^2(T_X).
\]
We obtain that the bottom left corner of the third page $E_3$ is the following.
\begin{equation*}
\begin{matrix}
\rH^3(T_X) & 0 & 0 & 0 \\
\coker d_2 & 0 & 0 & 0\\
\rH^1(T_X) & 0 & 0 & 0\\
\rH^0(T_X) & \ker d_2 & \rH^0(\cExt^2(\Omega_X, \cO_X)) & \rH^0(\cExt^3(\Omega_X, \cO_X))
\end{matrix}
\end{equation*}
In $E_3$ the only non-zero differential is
\[
d_3 \colon \rH^0(\cExt^2(\Omega_X, \cO_X)) \longrightarrow \rH^3(T_X).
\]
The bottom left corner of the fourth page $E_4$ is the following.
\begin{equation*}
\begin{matrix}
\coker d_3 & 0 & 0 & 0 \\
\coker d_2 & 0 & 0 & 0 \\
\rH^1(T_X) & 0 & 0 & 0 \\
\rH^0(T_X) & \ker d_2 & \ker d_3 & \rH^0(\cExt^3(\Omega_X, \cO_X))
\end{matrix}
\end{equation*}
From the fourth page on, the pieces of total degree $\leq 3$ do not change any more.
Therefore we have two short exact sequences:
\begin{gather*}
0 \longrightarrow \rH^1(T_X) \longrightarrow \Ext^1(\Omega_X, \cO_X) \longrightarrow \ker d_2 \longrightarrow 0, \\
0 \longrightarrow \coker d_2 \longrightarrow \Ext^2(\Omega_X, \cO_X) \longrightarrow \ker d_3 \longrightarrow 0.
\end{gather*}
These can be joined to construct the following long exact sequence.
\begin{align*}
0 &  &\longrightarrow & &\rH^1(T_X) & &\longrightarrow & &\Ext^1(\Omega_X, \cO_X) & &\longrightarrow & &\rH^0(\cExt^1(\Omega_X, \cO_X)) & &\overset{d_2}\longrightarrow \\
&  &\overset{d_2}\longrightarrow & &\rH^2(T_X) & &\longrightarrow & &\Ext^2(\Omega_X, \cO_X) & &\longrightarrow & &\rH^0(\cExt^2(\Omega_X, \cO_X))
\end{align*}

So far we did not use the assumption $\rH^2(T_X) = 0$. From this vanishing, via the long exact sequence above we deduce that the map $\Def_X \to \Defloc_X$ induces a surjection on tangent spaces and an injection on obstruction spaces.
\end{proof}

\subsection{Smoothings}
\label{sec:smoothings}
Here we discuss smoothability conditions for schemes of finite type over $\CC$. We will only consider the case of equidimensional schemes and we will refer the reader to \cite[\S29]{hartshorne_deformation_theory} for a more general treatment, which uses the Lichtenbaum--Schlessinger functors.

If $X$ is a proper scheme over $\CC$, a \emph{smoothing} of $X$ is a proper flat morphism $\cX \to B$ such that $B$ is an integral scheme of finite type over $\CC$ of positive dimension and there exists a closed point $b_0 \in B$ such that the fibre over $b_0$ is $X$ and all the other fibres are smooth. By restricting to a curve in $B$ and normalising it, we may require that the base $B$ is a smooth affine curve and that the maximal ideal corresponding to $b_0$ is principal. We say that $X$ is \emph{smoothable} if it admits a smoothing.

For every $n \geq 0$, set $S_n := \Spec \CC[t]/(t^{n+1})$. If $X$ is a scheme of finite type over $\CC$ with pure dimension $d$, then a \emph{formal smoothing} of $X$ is a formal deformation $\{ X_n \to S_n \}_n$ of $X$ over $\CC \pow{t}$ such that there exists $m$ such that $t^m$ is in the $d$th Fitting ideal of $\Omega_{X_m/S_m}$. 
We refer the reader to \cite[\S20.2]{eisenbud} for the definition and the properties of Fitting ideals.
We say that $X$ is \emph{formally smoothable} if it admits a formal smoothing. It is clear that if $X$ is formally smoothable, then every open subscheme of $X$ is formally smoothable.

\begin{remark} \label{rmk:potenze_di_t_nell_ideale_di_Fitting}
If $\{ X_n \to S_n \}_n$ is a formal deformation of $X$ over $\CC \pow{t}$ and $t^m$ is in the $d$th Fitting ideal of $\Omega_{X_m/S_m}$, then for all $n \geq m$ we have that $t^n$ is in the $d$th Fitting ideal of $\Omega_{X_n/S_n}$. 

The proof of this fact is as follows. We have
$\cO_{X_n} = \cO_{X_{n+1}} / t^{n+1} \cO_{X_{n+1}}$. Since the formation of Fitting ideals commutes with base change, we have the equality
\[
\mathrm{Fitt}_d(\Omega_{X_n/S_n}) = (\mathrm{Fitt}_d(\Omega_{X_{n+1}/S_{n+1}}) + t^{n+1} \cO_{X_{n+1}}) / t^{n+1} \cO_{X_{n+1}}.
\]
Therefore if $t^n \in \mathrm{Fitt}_d(\Omega_{X_n/S_n})$ then $t^n \in \mathrm{Fitt}_d(\Omega_{X_{n+1}/S_{n+1}}) + t^{n+1} \cO_{X_{n+1}}$, hence $t^{n+1} \in t \mathrm{Fitt}_d(\Omega_{X_{n+1}/S_{n+1}}) \subseteq \mathrm{Fitt}_d(\Omega_{X_{n+1}/S_{n+1}})$ as $t^{n+2} = 0$ in $\cO_{X_{n+1}}$.
\end{remark}

\begin{lemma} \label{lemma:smoothing_vs_formal_smoothing}
Let $X$ be a Cohen--Macaulay proper scheme over $\CC$ of pure dimension $d$. Let $B$ be a smooth curve over $\CC$, $b_0 \in B$ be a closed point, and $\pi \colon \cX \to B$ be a proper flat morphism such that the fibre over $b_0$ is $X$. Let $\xi $ be the formal $\frakm_{b_0}$-adic completion of $\pi$ at $b_0$, i.e.\ $\xi = \{ \cX \times_B \Spec \cO_{B, b_0} / \frakm_{b_0}^{n+1} \to \Spec \cO_{B, b_0} / \frakm_{b_0}^{n+1} \}_n$. Then:
\begin{enumerate}
\item if $\pi$ is a smoothing of $X$, then $\xi$ is a formal smoothing of $X$;
\item if $\xi$ is a formal smoothing of $X$, then there exists an open neighbourhood $B'$ of $b_0$ in $B$ such that $\cX \times_B B' \to B'$ is a smoothing of $X$.
\end{enumerate}
\end{lemma}

\begin{proof}
This proof comes from \cite[\href{https://stacks.math.columbia.edu/tag/0E7S}{Section 0E7S}]{stacks-project}.

Notice that $\xi$ does not change if we restrict $\pi$ to an open neighbourhood of $b_0$ in $B$.
Therefore, in order to prove the statements (1) and (2) we can arbitrarily restrict to an open neighbourhood of $b_0$ in $B$. Hence we may assume that $B$ is affine and the maximal ideal corresponding to the point $b_0$ is principal, generated by $t \in \cO_B$.

We consider the set $W \subseteq \cX$ made up of the points $ x \in \cX$ such that the local ring of the fibre $\cX_{\pi(x)}$ at $x$  is Cohen--Macaulay. By \cite[12.1.7]{ega_4_3}, $W$ is open in $\cX$. As $\pi$ is closed, $B \setminus \pi(\cX \setminus W)$ is an open neighbourhood of $b_0$ in $B$. Therefore, if we restrict $B$ to an open neighbourhood of $b_0$ in $B$, we may assume that all fibres of $\pi$ are Cohen--Macaulay. By \cite[\href{https://stacks.math.columbia.edu/tag/02NM}{Lemma 02NM}]{stacks-project}, we may assume that $\pi$ has relative dimension $d$.

Let $I \subseteq \cO_\cX$ be the $d$th Fitting ideal of $\Omega_{\cX / B}$. For each $n$, set
\[
S_n = \Spec \cO_{B, b_0} / \frakm_{b_0}^{n+1} = \Spec \cO_B / t^{n+1} \cO_B
\]
and $X_n = \cX \times_B S_n$; let $I_n \subseteq \cO_{X_n}$ be the $d$th Fitting ideal of $\Omega_{X_n / S_n}$. Since Fitting ideals commute with base change, we have $\cO_{X_n} = \cO_\cX / t^{n+1} \cO_\cX$ and $I_n = I \cO_{X_n} = (I + t^{n+1} \cO_{\cX})/ t^{n+1} \cO_{\cX}$. 

Since $\pi$ is flat of relative dimension $d$, the zero locus of $I$ is the singular locus of $\pi$. Moreover, the fibre over $b_0$ is the closed subset $\rV(t)$. Therefore, the fibre of $b_0$ is the unique singular fibre if and only if $t \in \sqrt{I}$.

(1) If $\pi$ is a smoothing, then there exists $m$ such that $t^m \in I$. Since $I_m = (I + t^{m+1} \cO_{\cX})/ t^{m+1} \cO_{\cX}$, this implies that $t^m \in I_m$. So $\xi$ is a formal smoothing.

(2) If $\xi$ is a formal smoothing, then $t^m \in I_m = (I + t^{m+1} \cO_{\cX})/ t^{m+1} \cO_{\cX}$ for some $m$. So in $\cO_\cX$ we have the equality $t^m = p + t^{m+1} q$, for some $p \in I$ and $q \in \cO_{\cX}$. Writing  $t^m (1-tq) = p$ and noticing that the function $1-tq$ does not vanish at the points of $X = \rV(t)$, we deduce that $t^m$ belongs to the stalk $I_x$ of $I$ at all points $x \in X$. This implies that $t^m$ lies in $I$ in an open neighbourhood $U$ of $X$ in $\cX$. Since $\pi$ is closed, by restricting $B$ to $B \setminus \pi(\cX \setminus U)$ we have $t^m \in I$. Therefore $\pi$ is a smoothing.
\end{proof}

\begin{proposition} \label{prop:smoothable_vs_formally_smoothable}
Let $X$ be a Cohen--Macaulay scheme proper over $\CC$.
\begin{enumerate}
\item If $X$ is smoothable, then every open subscheme of $X$ is formally smoothable.
\item Assume that $X$ is projective and $\rH^2(X, \cO_X) = 0$; if $X$ is formally smoothable, then $X$ is smoothable.
\end{enumerate}
\end{proposition}

\begin{proof}
We may assume that $X$ is connected. Therefore $X$ has pure dimension, say $d$.

(1) This follows immediately from Lemma~\ref{lemma:smoothing_vs_formal_smoothing} and from the fact that if $X$ is formally smoothable then every open subscheme of $X$ is formally smoothable.

(2) Set $d := \dim X$. Let $\xi = \{ X_n \to S_n \}_n$ be a formal smoothing of $X$, where $S_n$ is $\Spec \CC[t]/(t^{n+1})$ as usual. Let $m$ be such that $t^m$ is in the $d$th Fitting ideal of $\Omega_{X_m / S_m}$.

As $X$ is proper over $\CC$, the tangent space of $\Def_X$ has finite dimension, therefore $\Def_X$ has a hull $R \in \comp$. Let $\eta = \{ \eta_n \colon Y_n \to \Spec R / \frakm_R^{n+1} \}_n$ be the miniversal deformation of $X$.  By \cite[Proposition~6.51]{talpo_vistoli_deformation} or \cite[Theorem~2.5.13]{sernesi_deformation}, from $\rH^2(\cO_X) = 0$ we deduce that $\eta$ is effective, i.e.\ there exists a projective flat morphism $\mathscr{X} \to \Spec R$ whose $\frakm_R$-adic completion is $\eta$.

By a theorem of M.~Artin \cite[Theorem~1.6]{artin_algebraization_formal_moduli_1} (see also \cite[Theorem~21.3]{hartshorne_deformation_theory}), the morphism $\mathscr{X} \to \Spec R$ is algebraizable in the following sense: there exist a scheme $Z$ of finite type over $\CC$, a closed point $z_0 \in Z$, and a proper flat morphism $\cX \to Z$, with fibre $X$ over $z_0$, such that $R$ is the completion $\widehat{\cO}_{Z, z_0}$  of the local ring of $Z$ at $z_0$ and $\mathscr{X}$ is isomorphic, as $R$-schemes, to $\cX \times_Z \Spec R$. In particular, the miniversal deformation $\eta$ is the collection $\{ \cX \times_Z \Spec \cO_{Z, z_0} / \frakm_{z_0}^{n+1} \to \Spec \cO_{Z, z_0} / \frakm_{z_0}^{n+1} \}_n$.
The situation is summarised in the following cartesian squares, for all $n$.
\begin{equation*}
\begin{tikzcd}
Y_n \arrow[r, hook] \arrow[d, "\eta_n"] &
\mathscr{X} \arrow[r] \arrow[d] &
\cX \arrow[d] \\
\Spec R/ \frakm_R^{n+1} \arrow[r, hook] &
\Spec R = \Spec \widehat{\cO}_{Z, z_0} \arrow[r] &
Z
\end{tikzcd}
\end{equation*}

As $\eta$ is miniversal, there exists a local $\CC$-algebra homomorphism
\[
\phiv \colon \widehat{\cO}_{Z, z_0} = R \longrightarrow \CC \pow{t}\]
such that $\xi$ is induced by $\eta$ via $\phiv$, i.e.\ $X_n$ is isomorphic to $Y_n \times_{\Spec R / \frakm_R^{n+1}} S_n$ as $S_n$-schemes for every $n$. By another theorem of M.~Artin \cite[Corollary~2.5]{artin_algebraic_approximation}, the map $\phiv$ has an algebraic approximation up to order $m$ in the following sense: there exist a smooth affine curve $B$ over $\CC$ with a closed point $b_0 \in B$ and a $\CC$-morphism $f \colon B \to Z$ such that $f(b_0) = z_0$ and the completion
\[
\phiv' \colon \widehat{\cO}_{Z,z_0} = R \longrightarrow \widehat{\cO}_{B,b_0} = \CC \pow{t}
\]
of $f^{\#}_{b_0} \colon \cO_{Z,z_0} \to \cO_{B,b_0}$ satisfies the following property:
\begin{equation} \label{eq:congruenza}
\phiv \equiv \phiv' \ \text{ modulo } t^{m+1}.
\end{equation}
Let $\pi$ be the base change $\cX \times_Z B \to B$ along $f \colon B \to Z$. Let $\xi'$ be the formal $\frakm_{b_0}$-adic completion of $\pi$, i.e.\ $\xi' = \{ \cX \times_Z \Spec \cO_{B, b_0} / \frakm_{b_0}^{n+1} \to \Spec \cO_{B, b_0} / \frakm_{b_0}^{n+1} \}_n$. The two formal deformations $\xi$ and $\xi'$ of $X$ over $\CC\pow{t}$ are in general different, but they coincide up to order $m$ because of \eqref{eq:congruenza}. This implies that $t^m$ is in the $d$th Fitting ideal of the sheaf of K\"ahler differentials of $\cX \times_Z \Spec \cO_{B, b_0} / \frakm_{b_0}^{m+1} \to \Spec \cO_{B, b_0} / \frakm_{b_0}^{m+1}$. Therefore, $\xi'$ is a formal smoothing. By Lemma~\ref{lemma:smoothing_vs_formal_smoothing}, up to restrict $B$ to an open neighbourhood of $b_0$ in $B$, we have that $\pi \colon \cX \times_Z B \to B$ is a smoothing.
\end{proof}

The following theorem ensures that a projective scheme with formally smoothable isolated singularities is smoothable, provided that some local and cohomological conditions hold.

\begin{theorem} \label{thm:smoothable_sing_implies_globally_smoothable}
Let $X$ be a projective scheme over $\CC$ such that:
\begin{itemize}
\item $X$ is reduced and Cohen--Macaulay,
\item $X$ is either l.c.i.\ or normal,
\item $\rH^2(X,T_X) = 0$ and $\rH^2(X, \cO_X) = 0$,
\item $X$ has isolated singularities and for each singular point $x \in X$ there exists an open affine neighbourhood of $x$ which is formally smoothable.
\end{itemize}
Then $X$ is smoothable.
\end{theorem}

\begin{proof}
Set $d = \dim X$. Let $x_1, \dots, x_r$ be the singular points of $X$. Let $U_i$ be an affine open neighbourhood of $x_i$ in $X$ which is formally smoothable and such that $U_i \setminus \{ x_i \}$ is smooth. Let $\xi_i = \{ U_{i,n} \to S_n \}_n$ be a formal smoothing of $U_i$, where $S_n$ is $\Spec \CC[t]/(t^{n+1})$ as usual.

By Proposition~\ref{prop:no_local_to_global_obstructions}, from $\rH^2(T_X) = 0$ we deduce that the map $\Def_X \to \Defloc_X = \prod_{i=1}^r \Def_{U_i}$ is smooth. Therefore there exists a formal deformation $\xi = \{ X_n \to S_n \}_n$ of $X$ over $\CC \pow{t}$ such that for each $i$ the restriction of $\xi$ to $U_i$ is $\xi_i$, i.e.\ for all $n$ the restriction of $X_n$ to $U_i$ is $U_{i,n}$. By Remark~\ref{rmk:potenze_di_t_nell_ideale_di_Fitting} we have that $\xi$ is a formal smoothing of $X$. We conclude by Proposition~\ref{prop:smoothable_vs_formally_smoothable}.
\end{proof}

We now see some conditions that imply that a scheme is not smoothable.

\begin{proposition} \label{prop:two_conditions_for_non_formal_smoothability}
Let $X$ be a singular scheme of finite type over $\CC$ of pure dimension. Assume that at least one of the following conditions holds:
\begin{enumerate}
\item every infinitesimal deformation of $X$ is locally trivial,
\item the functor $\Def_X$ has an artinian hull.
\end{enumerate}
Then $X$ is not formally smoothable.
\end{proposition}

\begin{proof}
Set $d = \dim X$.

(1) Let $U$ be a singular affine open subscheme of $X$. Let $\{ X_n \to S_n \}_n$ be a formal deformation of $X$ over $\CC \pow{t}$. Let $U_n$ be the restriction of $X_n$ to $U$. By (1) we get that $U_n$ is isomorphic, as $S_n$-scheme, to the trivial deformation $U \times_{\Spec \CC} S_n$. Therefore $\mathrm{Fitt}_d(\Omega_{U_n/S_n}) = \mathrm{Fitt}_d(\Omega_{U/ \CC}) \cO_{U_n}$. As $U$ is singular, $\mathrm{Fitt}_d(\Omega_{U/ \CC}) \subsetneqq \cO_U$. This implies that $t^n \notin \mathrm{Fitt}_d(\Omega_{U_n/S_n})$.

(2) Let $R$ be the hull of $\Def_X$. Every formal deformation of $X$ over $\CC \pow{t}$ is induced by the miniversal one via a local $\CC$-algebra homomorphism $f \colon R \to \CC \pow{t}$. As every element in $\frakm_R$ is nilpotent and $\CC \pow{t}$ is a domain, the homomorphism $f$ factors as $R \onto R / \frakm_R = \CC \into \CC \pow{t}$. This implies that every formal deformation of $X$ over $\CC \pow{t}$ is trivial. Using a similar argument as in (1), we can prove that $X$ cannot have a formal smoothing.
\end{proof}

The following corollary, which is a direct consequence of Proposition~\ref{prop:locally_trivial_deformations}, Proposition~\ref{prop:smoothable_vs_formally_smoothable} and Proposition~\ref{prop:two_conditions_for_non_formal_smoothability}, gives some obstructions to the smoothability of a Cohen--Macaulay proper scheme.

\begin{cor} \label{cor:2conditions_implies_not_smoothable}
Let $X$ be a  Cohen--Macaulay scheme proper over $\CC$. Let $U \subseteq X$ be an open subscheme of $X$ such that $U$ is singular, reduced, and either l.c.i.\ or normal. If $\rH^0(U, \cExt^1(\Omega_U, \cO_U)) = 0$ or $\Def_U$ has an artinian hull (e.g.\ if $\Ext^1(\Omega_U, \cO_U) = 0$), then $X$ is not smoothable.
\end{cor}

\subsection{Invariants}
\label{sec:invariants}
Here we introduce a couple of invariants for Fano varieties.

The \emph{Hilbert series} of a Fano variety $X$ is the power series defined by its anti-plurigenera:
\[
\HilbS{X} := \sum_{m \geq 0} h^0(X, -mK_X) t^m \in \ZZ \pow{t}.
\]
The \emph{(anticanonical) degree} of a Fano variety $X$ is the positive rational number $(-K_X)^n$, where $n = \dim X$. If $X$ is Gorenstein, i.e.\ $K_X$ is Cartier, then the degree is an integer. The degree can be recovered from the Hilbert series because, up to a constant which depends on the dimension of $X$, it is the leading term of the Hilbert polynomial of $-K_X$.

The following proposition shows that the Hilbert series and the anticanonical degree are deformation invariants for Fano varieties with Gorenstein log terminal singularities.

\begin{proposition} \label{prop:invariants}
Let $S$ be a noetherian scheme over $\QQ$ and let $\pi \colon X \to S$ be a proper flat morphism whose geometric fibres are Fano varieties with Gorenstein log terminal singularities. Then the Hilbert series and the degree of the fibres are locally constant on $S$.
\end{proposition}

\begin{proof}

The morphism $\pi$ is a relatively Gorenstein. Therefore, by \cite[V.9.7]{residues_dualities}, the dualising sheaf $\omega_\pi$ is a line bundle on $X$ and its restriction to each fibre $X_s$ is $\cO_{X_s}(K_{X_s})$.

By Serre duality and Kawamata--Viehweg vanishing \cite[Theorem~2.70]{kollar_mori}, we get  $\rH^1(X_s, \cO_{X_s} (-mK_{X_s})) = 0$ for all $m \geq 0$ and $s \in S$. By cohomology and base change \cite[Theorem~III.12.11]{hartshorne}, for all $m \geq 0$, we get that the sheaf $\pi_* \omega_\pi^{\otimes - m}$ is locally free and has rank $h^0 (X_s, \cO_{X_s} (-mK_{X_s}))$ at the point $s \in S$. This implies that the Hilbert series of the fibres is locally constant on $S$.
\end{proof}

\section{Deformations of affine toric varieties} \label{sec:deformations_toric_singularities}

\subsection{Toric singularities}\label{sec:toric_singularities}

In this section we will consider deformations of toric singularities, that is affine toric varieties. We refer the reader to \cite{cox_toric_varieties, fulton_toric_varieties} for an introduction to  toric geometry.

If $X$ is an affine toric variety of dimension 2, then $X$ is a cyclic quotient surface singularity. There is extensive literature about deformations of this kind of singularities, e.g.\ \cite{riemenschneider, ksb, riemenschneider91, christophersen, stevens13, stevens91}. In particular, it is known that every affine toric variety of dimension 2 is smoothable \cite{artin_algebraic_construction}.

The study of the deformation theory of affine toric varieties of dimension at least 3 has been initiated by K.~Altmann \cite{altmann_one_parameter_families, altmann_infinitesimal_deformations_obstructions, altmann_versal_deformation, altmann_minkowski_sums, altmann_computation_tangent}. For example, he computed the tangent space of the deformation functor of an affine toric variety. We will not write down the explicit description of $\Ext^1(\Omega_X, \cO_X)$ when $X$ is an affine toric variety, but we will mention a consequence.

\begin{proposition}[{Altmann \cite[Corollary 6.5.1]{altmann_minkowski_sums}}] \label{prop:affine_smooth_codimension2_Qfactorial_codimension3_rigid}
If $X$ is a $\QQ$-Gorenstein affine toric variety which is smooth in codimension $2$ and $\QQ$-factorial in codimension $3$, then $X$ is rigid.
\end{proposition}

\begin{cor} \label{cor:QGorenstein_isol_dim4_rigid}
Every isolated $\QQ$-Gorenstein toric singularity of dimension $\geq 4$ is rigid.
\end{cor}

Now we need to do a brief detour on Minkowski sums. If $F_0, F_1, \dots, F_r$ are polytopes in a real vector space, their \emph{Minkowski sum} is the polytope
\[
F_0 + F_1 + \cdots + F_r := \{ v_0 + v_1 + \cdots + v_r \mid v_0 \in F_0, v_1 \in F_1, \dots, v_r \in F_r \}.
\]
When we have $F = F_0 + F_1 + \cdots + F_r$, we say that we have a \emph{Minkowski decomposition} of the polytope $F$. We consider Minkowski decompositions up to translation:
for instance, we consider the Minkowski decomposition $F = (v + F_0) + (-v + F_1)$ to be
equivalent to $F = F_0 + F_1$ for every vector $v$. Moreover, in what follows we require that
the summands $F_j$ are lattice polytopes, i.e.\ their vertices belong to a fixed lattice.

Altmann~\cite{altmann_minkowski_sums} has noticed that certain Minkowski decompositions induce deformations of affine toric varieties. In \S\ref{sec:affine_cone_dP7} we will see an example of this fact.
For the proof we refer the reader to the original reference \cite{altmann_minkowski_sums} and  to  \cite{mavlyutov, petracci_mavlyutov}.

\bigskip

Now let us concentrate on Gorenstein toric singularities.
They are associated to lattice polytopes of dimension one less than the dimension of the singularity.
More precisely, let $F$ be a lattice polytope of dimension $n-1$ in a lattice $\overline{N}$ of rank $n-1$ and let $U_F$ be the affine toric variety associated to the cone $\sigma_F = \RR_{\geq 0} (F \times \{ 1 \})$ in the lattice $N := \overline{N} \oplus \ZZ$, i.e.\ $U_F = \Spec \CC[\sigma_F^\vee \cap M]$, where $M = \overline{M} \oplus \ZZ$ is the dual of $N$ and $\sigma_F^\vee$ is the dual cone of $\sigma_F$.
We have that $U_F$ has dimension $n$ and is Gorenstein.
All Gorenstein affine toric varieties without torus factors arise in this way from a lattice polytope. The isomorphism class of $U_F$ does not change if we change $F$ via an affine transformation in $\overline{N} \rtimes \mathrm{GL}(\overline{N}, \ZZ)$.

As usual in toric geometry, the geometric properties of $U_F$ can be deduced from the combinatorial properties of $F$. For instance:
\begin{itemize}
\item $U_F$ is smooth in codimension $k$ if and only if all faces of $F$ with dimension $<k$ are standard simplices;
\item $U_F$ is $\QQ$-factorial in codimension $k$ if and only if all faces of $F$ with dimension $<k$ are simplices.
\end{itemize}
It is always the case that $U_F$ is smooth in codimension $1$ and $\QQ$-factorial in codimension $2$.

If $F$ is a segment of lattice length $m+1$, then $U_F$ is the $A_m$ surface singularity $\Spec \CC[x,y,z]/(xy-z^{m+1})$. This is an isolated hypersurface singularity, therefore it is very easy to write down the miniversal deformation: $xy = z^{m+1} + t_m z^{m-1} + \cdots + t_1$ over $\CC \pow{t_1, \dots, t_m}$. It is clear that this singularity is smoothable.

If $F$ is a lattice polygon, then the affine toric threefold $U_F$ has the following properties:
\begin{itemize}
\item $U_F$ has, at most, an isolated singularity if and only if the edges of $F$ are unitary, i.e.\ have lattice length $1$;
\item $U_F$ is $\QQ$-factorial if and only if $F$ is a triangle.
\end{itemize}

Now we provide some examples of lattice polygons and their corresponding toric Gorenstein affine threefolds.

\begin{example} \label{ex:polygons}
A lattice polygon $F$ is called a \emph{standard square} if it is a quadrilateral such that all its lattice points are vertices, or equivalently if it is $\ZZ^2 \rtimes \mathrm{GL}_2(\ZZ)$-equivalent to $\conv{(0,0), (1,0), (1,1), (0,1)} \subseteq \RR^2$. If $F$ is a standard square, then $U_F$ is the ordinary double point (i.e.\ node) $\Spec \CC[x,y,z,w]/(xy-zw)$. This singularity is clearly smoothable as it is a hypersurface singularity. Its miniversal deformation is given by $xy-zw=t$ over $\CC \pow{t}$.

A lattice polygon $F$ is called a \emph{standard triangle} if it is a triangle such that all its lattice points are vertices, or equivalently if it is $\ZZ^2 \rtimes \mathrm{GL}_2(\ZZ)$-equivalent to $\conv{(0,0), (1,0), (0,1)} \subseteq \RR^2$. $F$ is a standard triangle if and only if $U_F$ is isomorphic to $\AA^3$.

If $m \geq 1$, then a lattice polygon $F$ is called an \emph{$A_m$-triangle} if it is a triangle such that there are no interior lattice points and the edges have lattice lengths $1,1,m+1$, respectively. Equivalently, a polygon is an $A_m$-triangle if and only if it is $\ZZ^2 \rtimes \mathrm{GL}_2(\ZZ)$-equivalent to $\conv{(0,0), (m+1,0), (0,1)} \subseteq \RR^2$. If $F$ is an $A_m$-triangle, then $U_F$ is the $cA_m$-singularity $\Spec \CC[x,y,z,w]/(xy-z^{m+1})$.  This singularity is clearly smoothable as it is a hypersurface singularity.
\end{example}

\begin{figure}
\centering
\includegraphics[width=10cm]{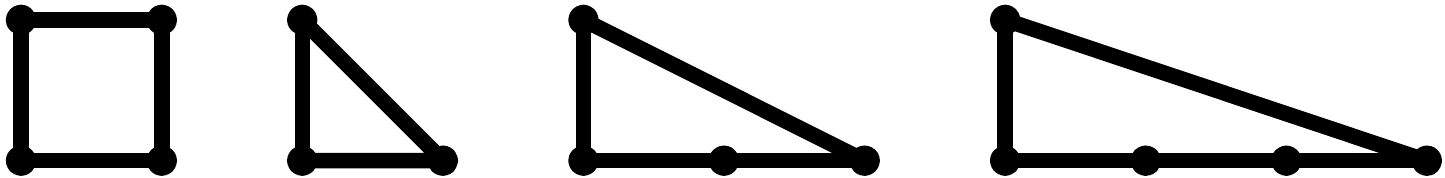}
\caption{A standard square, a standard triangle, an $A_1$-triangle and an $A_2$-triangle.}
\label{fig:poligoni}
\end{figure}

Altmann \cite{altmann_versal_deformation} explicitly constructed the miniversal deformation of an isolated Gorenstein toric singularity of dimension 3. 
(By Corollary~\ref{cor:QGorenstein_isol_dim4_rigid} it is trivial to construct the miniversal deformation of an isolated Gorenstein toric singularity of dimension $\geq 4$.)
A consequence of his construction is the following description of the irreducible components of the base of the miniversal deformation.

\begin{theorem}[{Altmann \cite{altmann_versal_deformation}}]
Let $F$ be a lattice polygon with unitary edges and let $U_F$ be the corresponding isolated Gorenstein toric singularity of dimension $3$. Let $R$ be the hull of $\Def_{U_F}$. Then there exists a one-to-one correspondence between minimal primes of $R$ and maximal Minkowski decompositions of $F$.
Moreover, if a minimal prime $\frakp \subset R$ corresponds to the maximal Minkowski decomposition $F = F_0 + F_1 + \cdots + F_r$, then $r = \dim R/ \frakp$.
\end{theorem}

\begin{cor}
\label{cor:Gor_isol_3dim_Minkowskindec_iff_artinian}
Let $F$ be a lattice polygon with unitary edges and let $U_F$ be the associated isolated Gorenstein toric singularity of dimension $3$. Then $\Def_{U_F}$ has an artinian hull if and only if $F$ is Minkowski indecomposable.
\end{cor}

\subsection{The affine cone over the del Pezzo surface of degree 7}
\label{sec:affine_cone_dP7}

Here we study an explicit example of what has been considered in \S\ref{sec:toric_singularities}.
In the lattice $\overline{N} = \ZZ^2$ consider the pentagon
\begin{equation} \label{eq:pentagon}
F = \conv{
\vectortwo{1}{0},
\vectortwo{1}{1},
\vectortwo{0}{1},
\vectortwo{-1}{0},
\vectortwo{0}{-1}
} \subseteq \overline{N}_\RR,
\end{equation}
which is depicted on the left of Figure~\ref{fig:Minkow}.
The toric variety associated to the face fan of $F$ is the smooth del Pezzo surface of degree $7$, which is denoted by $\dP_7$ and is the blow up of $\PP^2$ in 2 distinct points. The anticanonical map of $\dP_7$ is a closed embedding into $\PP^7$.

\begin{figure}
\centering
\def\svgwidth{10cm}
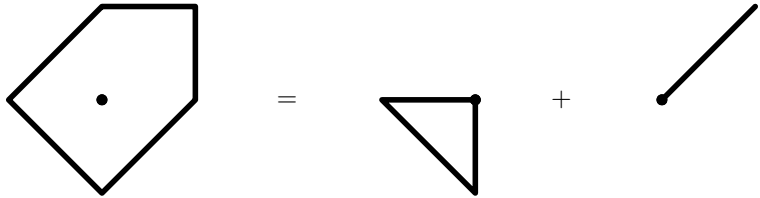
\caption{The Minkowski decomposition \eqref{eq:Minkowski_decomposition_pentagon} of the pentagon $F$ in \eqref{eq:pentagon}.}
\label{fig:Minkow}
\end{figure}

Now we put the pentagon $F$ at height $1$ in the lattice $N = \overline{N} \oplus \ZZ$ and we consider the cone over it:
\begin{equation*}
\sigma_F = \cone{\vector{1}{0}{1}, \vector{1}{1}{1}, \vector{0}{1}{1}, \vector{-1}{0}{1}, \vector{0}{-1}{1}} \subseteq \overline{N}_\RR \oplus \RR.
\end{equation*}
The affine toric variety $U_F = \Spec \CC[\sigma_F^\vee \cap (\overline{M} \oplus \ZZ)]$ is the affine cone over the anticanonical embedding of $\dP_7$ and has an isolated Gorenstein canonical non-terminal singularity at the vertex of the cone.

Altmann \cite[(9.1)]{altmann_versal_deformation} shows that the hull of $\Def_{U_F}$ is $\CC \pow{t_1, t_2}/(t_1^2, t_1 t_2)$, which is a line with an embedded point. The reduction of the miniversal deformation, i.e.\ the base change to the reduction of the hull, is induced by the unique maximal Minkowski decomposition of the pentagon $F$ in the following way.

In the lattice $\overline{N}$ we have the Minkowski decomposition
\begin{equation}
\label{eq:Minkowski_decomposition_pentagon}
F = \conv{
\vectortwo{0}{0},
\vectortwo{-1}{0},
\vectortwo{0}{-1}
}
+
\conv{
\vectortwo{0}{0},
\vectortwo{1}{1}
},
\end{equation}
which is illustrated in Figure~\ref{fig:Minkow}.
Following \cite[(3.4)]{altmann_minkowski_sums}, in the lattice $\tilde{N} = \overline{N} \oplus \ZZ e_1 \oplus \ZZ e_2$ we construct the cone
\begin{equation*}
\label{eq:sigma_tilde}
\tilde{\sigma} = \cone{
\vectorfour{0}{0}{1}{0},
\vectorfour{-1}{0}{1}{0},
\vectorfour{0}{-1}{1}{0},
\vectorfour{0}{0}{0}{1},
\vectorfour{1}{1}{0}{1}
}
\subseteq \tilde{N}_\RR.
\end{equation*}
Notice that the the first three rays of $\tilde{\sigma}$ come from the vertices of the first summand of $F$ in \eqref{eq:Minkowski_decomposition_pentagon}, whereas the last two rays of $\tilde{\sigma}$ come from the vertices of the second summand of $F$ in \eqref{eq:Minkowski_decomposition_pentagon}.
Let $\tilde{U} = \Spec \CC [\tilde{\sigma}^\vee \cap \tilde{M}]$ be the affine toric variety associated to the cone $\tilde{\sigma}$, where $\tilde{M}$ denotes the dual of $\tilde{N}$. One can prove that $\tilde{U}$ has only an isolated terminal Gorenstein singularity. Let $f_1$ and $f_2$ be the regular functions on $\tilde{U}$ associated to the characters $(0,0,1,0) \in \tilde{M}$ and $(0,0,0,1) \in \tilde{M}$, respectively. The variety $U_F$ is the zero locus of the function $f_1 - f_2$, i.e.\ we have a cartesian diagram
\begin{equation} \label{eq:square_deformation_U_pentagon}
\xymatrix{
U_F \ar[d] \ar[r] & \tilde{U} \ar[d]^{\pi} \\
\Spec \CC \ar[r] & \AA^1_\CC
}
\end{equation}
where $\pi$ is given by the function $f_1-f_2$ and the bottom morphism is given by the origin of $\AA^1_\CC$.
Since $f_1 - f_2$ is not constant and $\AA^1_\CC$ is regular of dimension 1, the morphism $\pi$ is flat. The reduction of the miniversal deformation of $U_F$ is the formal deformation of $U_F$ over $\CC\pow{t}$ obtained from the square \eqref{eq:square_deformation_U_pentagon} by base change via $\Spec \CC[t] / (t^{n+1}) \into \Spec \CC[t] = \AA^1_\CC$ for all $n$. The following proposition shows that this is a formal smoothing.

\begin{proposition} \label{prop:smoothing_U_dP7}
Let $F$ be the pentagon defined in \eqref{eq:pentagon} and let $U_F$ be the corresponding Gorenstein toric threefold singularity.
Then the collection of the base change of $\pi$  in \eqref{eq:square_deformation_U_pentagon} via $\Spec \CC[t] / (t^{n+1}) \to \Spec \CC[t] = \AA^1_\CC$ for all $n$ is a formal smoothing of $U_F$. In particular, $U_F$ is formally smoothable.
\end{proposition}

\begin{proof}
We want to study the closed fibres of $\pi$. The fibre over the origin of $\AA^1_\CC$ is $U_F$. 
Let us fix $\lambda \in \CC \setminus \{ 0 \}$ and we consider the fibre $\pi\inv (\lambda)$ of $\pi$ over the closed point $(t-\lambda)$ of $\AA^1_\CC$ corresponding to $\lambda$.
We consider the subcone $\tau_1$ (resp.\ $\tau_2$) of
$\tilde{\sigma}$ that is generated by the first three (resp.\ last two) rays of $\tilde{\sigma}$.
We consider the affine toric variety $W_j = \Spec \CC [\tau_j^\vee \cap \tilde{M}]$, for $j=1,2$.
We have that $W_1$ and $W_2$ are open subschemes of $\tilde{U}$.

We have that $W_1$ is the open subset of $\tilde{U}$ where the function $f_2$ does not vanish, i.e.\ $W_1 = \{ f_2 \neq 0 \} \subseteq \tilde{U}$, and analogously $W_2 = \{ f_1 \neq 0 \} \subseteq \tilde{U}$.
It is clear that there is an isomorphism $W_1 \simeq \AA^3 \times \GG_\rmm$ with respect to which the function $f_1 \vert_{W_1}$ becomes a projection onto a $\AA^1$-factor in $\AA^3$ and the function $f_2 \vert_{W_1}$ becomes the projection onto the $\GG_\rmm$-factor.
There is also an isomorphism $W_2\simeq \AA^2 \times \GG_\rmm^2$ with respect to which the function $f_1 \vert_{W_2}$ becomes a projection onto a $\GG_\rmm$-factor and the function $f_2 \vert_{W_2}$ becomes  a projection onto an $\AA^1$-factor.

Now $\pi\inv(\lambda) \cap W_1 = \{f_1= f_2 + \lambda \} \cap W_1$ is isomorphic to $\AA^2 \times \GG_\rmm$ and
$\pi\inv(\lambda) \cap W_2 = \{ f_2 = f_1 - \lambda \} \cap W_2$ is isomorphic to $\AA^1 \times \GG_\rmm^2$.
Since $\lambda \neq 0$, it is clear that $\pi\inv(\lambda) \subseteq W_1 \cup W_2$. Therefore we have proved that $\pi\inv(\lambda)$ is smooth.

One can also show that the generic fibre of $\pi$ is smooth over the generic point of $\AA^1_\CC$. In order to prove this, it is enough to base change to the spectrum of the field $\CC(t)$ of rational function of $\AA^1_\CC$ and pursue a similar argument, which deals with toric varieties over the field $\CC(t)$.

In particular, $\pi$ is flat of relative dimension $3$ and has Cohen--Macaulay fibres. As in the proof of Lemma~\ref{lemma:smoothing_vs_formal_smoothing}, from the fact that all non-special fibres of $\pi$ are smooth we can deduce that $\pi$ induces a formal smoothing of $U_F$.
\end{proof}

\section{Deformations of toric Fano varieties}
\label{sec:deformations_toric_fano}

\subsection{Fano polytopes}
\label{sec:Fano_polytopes}

Fano polytopes are the combinatorial-polyhedral avatars of toric Fano varieties.

\begin{definition}
A polytope $P$ in a lattice $N$ of rank $n$ is called \emph{Fano} if
\begin{itemize}
\item $P$ has dimension $n$,
\item the origin $0$ lies in the interior of $P$,
\item the vertices of $P$ are primitive lattice elements of $N$.
\end{itemize}
If $P$ is a Fano polytope, we denote by $X_P$ the complete toric variety associated to the spanning fan (also called the face fan) of $P$.
\end{definition}

If $P$ is a Fano polytope, then $X_P$ is a Fano variety.
All toric Fano varieties arise in this way from a Fano polytope \cite[\S8.3]{cox_toric_varieties}.
The variety $X_P$ is Gorenstein, i.e.\ its (anti)canonical divisor is Cartier, if and only if $P$ is reflexive, i.e.\ the facets of $P$ lie on hyperplanes with height $1$ with respect to the origin. 
The maximal toric affine charts of $X_P$ (or equivalently the torus-fixed points of $X_P$) are in one-to-one correspondence with the facets of $P$.
If $n$ is the dimension of $P$, for every $0 \leq k \leq n$ there is a one-to-one correspondence between the $k$-dimensional torus-orbits of $X_P$ and the $(n-k-1)$-dimensional faces of $P$.

Fano polytopes of small dimension with specific properties have been classified \cite{al_log_del_pezzo, al_terminal,al_canonical, kreuzer_skarke_reflexive_4topes, kreuzer_skarke_reflexive_3topes, batyrev_fano_3folds, batyrev_4folds, sato_5, sato_2000, nill_obro, kreuzer_nill, watanabe2, obro}.
We refer the reader to \cite{kasprzyk_nill_fano_polytopes} for a survey on the classification of Fano polytopes.

\subsection{Two sufficient conditions for non-smoothability}
\label{sec:toric_rigid_singularity_then_non_smoothable}

It is an open problem to understand whether an arbitrary toric Fano variety is smoothable.
Here we state a couple of conditions that forbid the smoothability. Both conditions on a toric Fano variety $X$ are based on the existence of an open affine singular subscheme $U$ such that $U$ is not formally smoothable.

\begin{theorem}
\label{thm:non_smoothable_face_implies_non_smoothable}
Let $N$ be a lattice, let $P$ be a Fano polytope in $N$, and let $X$ be the toric Fano variety associated to the spanning fan of $P$.
Assume that there exists a face $F$ of $P$ which satisfies the following conditions:
\begin{enumerate}
\item[(i)] for each $1$-face $F'$ of $F$, there exists a basis of $N$ which contains the two vertices of $F'$,
\item[(ii)] each $2$-face of $F$ is a triangle,
\item[(iii)] there exists no basis of $N$ which contains all the vertices of $F$.
\end{enumerate}
Then $X$ is not smoothable.
\end{theorem}

\begin{proof}
Let $U$ be the affine toric open subscheme of $X$ associated to the cone spanned by the face $F$. The condition (i) means that $U$ is smooth in codimension 2. The condition (ii) means that $U$ is $\QQ$-factorial in codimension 3. Therefore $U$ is rigid by Proposition~\ref{prop:affine_smooth_codimension2_Qfactorial_codimension3_rigid}. The condition (iii) implies that $U$ is singular.
Therefore, by Corollary~\ref{cor:2conditions_implies_not_smoothable}, $X$ is not smoothable.
\end{proof}

If $P$ is a reflexive polytope of dimension 3, then the theorem above applies if there exists a triangular facet $F$ with unitary edges and such that it is not a standard triangle. Below we see that we can relax the condition of $F$ being triangular to $F$ being Minkowski-indecomposable.

\begin{proposition} \label{prop:3folds_non_smoothable_isolated_sing}
Let $P$ be a reflexive polytope of dimension 3 and let $X$ be the toric Fano threefold associated to the spanning fan of $P$.
Assume that there exists a facet $F$ of $P$ such that:
\begin{enumerate}
\item[(i)] $F$ has unitary edges,
\item[(ii)] $F$ is Minkowski-indecomposable,
\item[(iii)] $F$ is not a standard triangle (i.e.\ the vertices of $F$ do not form a basis of the lattice).
\end{enumerate}
Then $X$ is not smoothable.
\end{proposition}

\begin{proof}
The proof is very similar to the proof of Theorem~\ref{thm:non_smoothable_face_implies_non_smoothable}.
Let $U$ be the affine toric open subscheme of $X$ associated to the cone spanned by $F$.
The conditions (i) and (iii) means that $U$ has an isolated singularity. Since $P$ is reflexive, $U$ is Gorenstein.
By Corollary~\ref{cor:Gor_isol_3dim_Minkowskindec_iff_artinian}, from (ii) we deduce that $\Def_U$ has an artinian hull.
Therefore, by Proposition~\ref{prop:two_conditions_for_non_formal_smoothability}, $U$ is not formally smoothable.
By Corollary~\ref{cor:2conditions_implies_not_smoothable}, $X$ is not smoothable.
\end{proof}

\subsection{Rigidity}
\label{sec:rigidity}

Here we will see that if a toric Fano variety has very mild singularities then it is rigid.

\begin{lemma} \label{lemma:no_higher_coh_tangent_Fano_toric}
Let $X$ be a toric Fano variety. Then $\rH^i(X, T_X) = 0$ for each $i \geq 1$. In particular, all locally trivial deformations of $X$ are trivial.
\end{lemma}

\begin{proof}
Set $n = \dim X$. Consider the smooth locus $j \colon U \into X$. Let $D$ be the toric boundary of $X$.
The sheaves $T_X$ and $(j_* \Omega_U^{n-1} \otimes \cO_X(D))^{\vee \vee}$ are reflexive on $X$ and their restrictions to $U$ coincide, because $U$ is smooth and $T_U$ is isomorphic to $\Omega_{U}^{n-1} \otimes \omega_U^\vee$.
Therefore, since the complement of $U$ has codimension at least $2$, by \cite[Proposition~1.6]{hartshorne_reflexive} we have that $T_X$ is isomorphic to $(j_* \Omega_U^{n-1} \otimes \cO_X(D))^{\vee \vee}$.
Since $D$ is ample, we conclude by Bott--Steenbrink--Danilov vanishing \cite[Theorem~9.3.1]{cox_toric_varieties} (see also \cite{buch_thomsen, mustata_vanishing, fujino_vanishing}).
\end{proof}

An immediate consequence of the lemma above is the following result.

\begin{proposition}
Every smooth toric Fano variety is rigid.
\end{proposition}

This result was originally proved by Bien and Brion \cite{bien_brion}.
Later de Fernex and Hacon \cite{de_fernex_hacon} proved the rigidity of $\QQ$-factorial terminal toric Fano varieties. The following theorem, due to Totaro, is the most general rigidity theorem for toric Fano varieties of which we are aware.

\begin{theorem}[{Totaro \cite[Theorem 5.1]{totaro_jumping}}] \label{thm:totaro_rigidity}
A Fano toric variety which is smooth in codimension $2$ and $\QQ$-factorial in codimension $3$ is rigid.
\end{theorem}

\begin{proof}
By Lemma~\ref{lemma:no_higher_coh_tangent_Fano_toric}, $\rH^1(T_X) = 0$. By Proposition~\ref{prop:affine_smooth_codimension2_Qfactorial_codimension3_rigid}, the sheaf $\cExt^1(\Omega_X, \cO_X)$ is zero. From the five term exact sequence of $\Ext$, which is written in the proof of Proposition~\ref{prop:locally_trivial_deformations}, we deduce that $\Ext^1(\Omega_X, \cO_X)$ is zero.
\end{proof}

If $P$ is a Fano polytope, then $X_P$ satisfies the hypotheses of this theorem if and only if all $2$-faces of $P$ are triangles and each edge, i.e.\ $1$-face, of $P$ has lattice length $1$ and is contained in some hyperplane which has height $1$ with respect to the origin.

\begin{cor} \label{cor:fourfolds_isolated_rigid}
Let $X$ be a toric Fano variety of dimension $\geq 4$. If $X$ has isolated singularities, then $X$ is rigid.
\end{cor}

In \S\ref{sec:toric_dP} and \S\ref{sec:toric_fano_3folds_with_isolated_singul} we will study deformations of toric Fanos with isolated singularities and of dimension 2 or 3.

\subsection{Toric del Pezzo surfaces}
\label{sec:toric_dP}

A \emph{del Pezzo surface} is a Fano variety of dimension 2.
A toric del Pezzo surface is associated to a Fano polygon, which is a Fano polytope of dimension 2.

\begin{theorem} \label{thm:toric_dP}
Every toric del Pezzo surface is smoothable.
\end{theorem}

\begin{proof}
Let $X$ be an arbitrary toric del Pezzo surface. It is well known that $X$ is a normal Cohen--Macaulay projective variety. By Demazure vanishing \cite[Theorem~2.9.3]{cox_toric_varieties}, $\rH^2(\cO_X) = 0$.
By Lemma~\ref{lemma:no_higher_coh_tangent_Fano_toric}, $\rH^2(T_X) = 0$.
Since $X$ is normal and of dimension 2, $X$ has isolated singularities.
By Theorem~\ref{thm:smoothable_sing_implies_globally_smoothable} it is enough to check that the singularities of $X$ are formally smoothable.

The singularities of $X$ are cyclic quotient surface singularities. This kind of singularities is always smoothable; indeed, it is enough to pick the Artin component of the base of the miniversal deformation \cite{artin_algebraic_construction}.
\end{proof}

\begin{remark}
When the canonical divisor of a normal variety $X$ is not Cartier, flat deformations of $X$ are too wild for hoping to study moduli of varieties. For a normal $\QQ$-Gorenstein non-Gorenstein variety $X$ one should consider a subfunctor of $\Def_X$ which is made up of the deformations of $X$ in which the canonical divisor deforms well. This is the theory of $\QQ$-Gorenstein deformations, developed by Koll\'ar--Shepherd-Barron \cite{ksb} (see also \cite{hacking_duke,
hassett_abramovich, lee_nakayama, altmann_kollar}).

In the context of $\QQ$-Gorenstein deformations the analogous statement of Theorem~\ref{thm:toric_dP} is false: there exist non-Gorenstein toric del Pezzo surfaces which cannot be deformed via $\QQ$-Gorenstein deformations to a smooth del Pezzo surface, e.g.\ the weighted projective space $\PP(1,1,3)$. Nonetheless, it is true that for $\QQ$-Gorenstein deformations of del Pezzo surfaces there are no local-to-global obstructions \cite[Lemma~6]{procams}. Therefore, a del Pezzo surface is $\QQ$-Gorenstein smoothable if and only if its singularities are $\QQ$-Gorenstein smoothable.

Since the main focus of this note is the study of deformations of Gorenstein toric Fano threefolds, we will omit to discuss the theory of $\QQ$-Gorenstein deformations. We refer the reader to \cite{hacking_prokhorov, prince_smoothing} for the study of toric del Pezzo surfaces which have $\QQ$-Gorenstein smoothings.
\end{remark}

\subsection{Toric Fano threefolds with isolated singularities}
\label{sec:toric_fano_3folds_with_isolated_singul}

\begin{theorem} \label{thm:fano_3fold_isolat_sing_smoothable}
Let $X$ be a toric Fano variety of dimension $3$ with isolated singularities.
Then $X$ is smoothable if and only if its singularities are formally smoothable.
\end{theorem}

\begin{proof}
By Proposition~\ref{prop:smoothable_vs_formally_smoothable}, if $X$ is smoothable then its singularities are formally smoothable.
Conversely, suppose that the singularities of $X$ are formally smoothable.
Then we argue as in the proof of Theorem~\ref{thm:toric_dP}: $X$ is a normal Cohen--Macaulay projective variety with $\rH^2(\cO_X) = 0$, by \cite[Theorem~2.9.3]{cox_toric_varieties}, and $\rH^2(T_X) = 0$, by Lemma~\ref{lemma:no_higher_coh_tangent_Fano_toric}.
By Theorem~\ref{thm:smoothable_sing_implies_globally_smoothable}, $X$ is smoothable.
\end{proof}

\begin{cor} \label{cor:fano_3fold_nodes}
Let $P$ be a reflexive polytope of dimension 3 and let $X$ be the toric Fano threefold associated to the spanning fan of $P$.
If each facet of $P$ is either a standard triangle or a standard square (see the definition in Example~\ref{ex:polygons}), then $X$ is smoothable.
\end{cor}

\begin{proof}
By Example \ref{ex:polygons} we have that the singularities of $X$ are at most ordinary double points (i.e.\ nodes). These singularities are formally smoothable. By Theorem~\ref{thm:fano_3fold_isolat_sing_smoothable} we conclude.
\end{proof}

The proof of this corollary is essentially a specific case of \cite[\S4.a]{friedman_threefold_node}.
The corollary could have been deduced also from a more general result by Namikawa according to which every Fano threefold with Gorenstein terminal singularities is smoothable \cite{namikawa_fano_3folds}. The smooth Fano threefolds which are the smoothings of the toric Fano threefold appearing in Corollary~\ref{cor:fano_3fold_nodes} have been studied by Galkin \cite{galkin_small}.

For $d \in \{ 6,7\}$, let $\mathrm{dP}_d$ be the smooth del Pezzo surface of degree $d$; it is toric.
The complete anticanonical linear system on $\mathrm{dP}_d$ induces a closed embedding $\mathrm{dP}_d \into\PP^d$.
We consider the projective cone $C(\mathrm{dP}_d) \subseteq \PP^{d+1}$ over this embedding; we have that $C(\mathrm{dP}_d)$ is a toric Fano threefold with a Gorenstein canonical non-terminal isolated singularity. In \S\ref{sec:projective_cone_dP7} we will see that $C(\mathrm{dP}_7)$ is smoothable. In \cite{petracci_roma} it is shown that $C(\mathrm{dP}_6)$ has two smoothings (see also \cite[Example~3.3]{jahnke_radloff}).

\subsection{The projective cone over the del Pezzo surface of degree 7}
\label{sec:projective_cone_dP7}

Here we study the deformations of an explicit toric Fano threefold with an isolated Gorenstein non-terminal singularity.

Fix the lattice $\overline{N} = \ZZ^2$. Consider the pentagon $F \subseteq \overline{N}_\RR$ defined in \eqref{eq:pentagon}, imagine to put it into the plane $\overline{N}_\RR \times \{1 \}$ in $\overline{N}_\RR \oplus \RR \simeq \RR^3$, and create the pyramid over it with apex at the point $(0,0,-1)$: this is the polytope 
\begin{equation} \label{eq:pyramid}
P = \conv{\vector{1}{0}{1},
\vector{1}{1}{1},
\vector{0}{1}{1},
\vector{-1}{0}{1},
\vector{0}{-1}{1},
\vector{0}{0}{-1}}
\end{equation}
in the lattice $\overline{N} \oplus \ZZ$ and is depicted in Figure~\ref{fig:piramide}. It is clear that $P$ is a Fano polytope.

\begin{figure}[b]
\centering
\includegraphics[width=6cm]{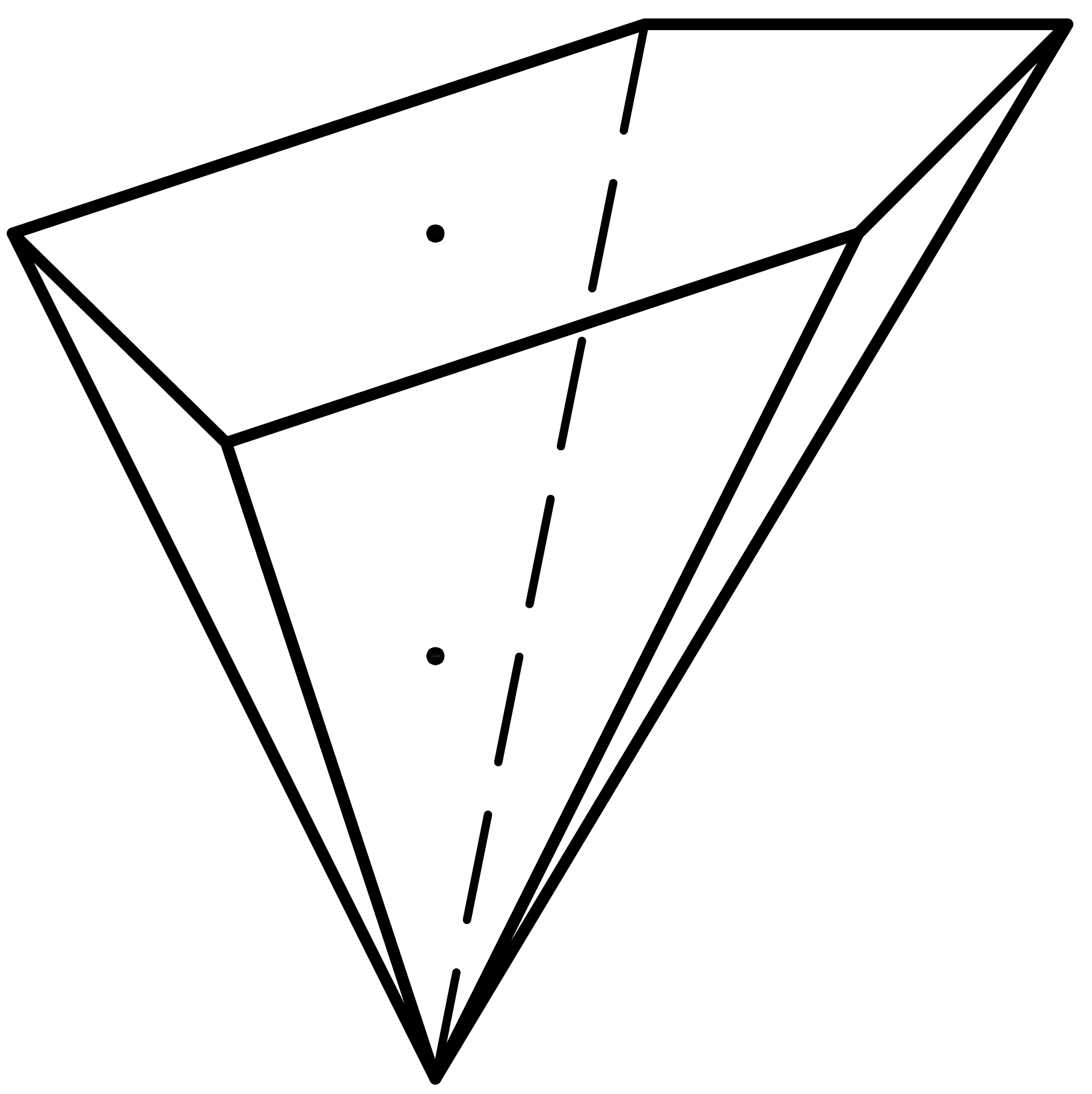}
\caption{The 3-dimensional lattice polytope $P$ defined in \eqref{eq:pyramid} and associated to the projective cone over the del Pezzo surface of degree 7.}
\label{fig:piramide}
\end{figure}

Let $X$ be the toric variety associated to the spanning fan of $P$. Then $X$ is the projective cone over the anticanonical embedding of the smooth del Pezzo surface of degree 7.
The affine toric variety $U_F$ considered in \S\ref{sec:affine_cone_dP7} is the affine open toric subscheme of $X$ associated to the pentagonal facet $F$ of $P$.
We have that $X$ is a Fano threefold with an isolated non-terminal  canonical Gorenstein singularity at the vertex of the cone.

\begin{proposition}
Let $X$ be the toric Fano threefold associated to the polytope $P$ in \eqref{eq:pyramid}, i.e.\ $X$ is the projective cone over the anticanonical embedding of the smooth del Pezzo surface of degree 7.
Then $X$ is smoothable and can be deformed to the smooth Fano threefold $\PP (\cO_{\PP^2} \oplus \cO_{\PP^2}(1))$.
\end{proposition}

\begin{proof}
By Proposition~\ref{prop:smoothing_U_dP7}, $X$ has an isolated singularity which is formally smoothable.
By Theorem~\ref{thm:fano_3fold_isolat_sing_smoothable} we know that $X$ is smoothable. We need to know to which smooth Fano threefold $X$ can be deformed.

From toric geometry \cite[Theorem~13.4.3]{cox_toric_varieties}, we have that the anticanonical degree $(-K_X)^3$ is the normalised volume of the polar polytope of $P$, which is 56 in this case. Since $X$ has Gorenstein canonical singularities, by Proposition~\ref{prop:invariants} we have that the anticanonical degree is preserved in the smoothing.
By inspecting the list of smooth Fano threefolds (see \cite{iskovskih_fano_1, iskovskih_fano_2, mori_mukai_fano, mori_mukai_fano_erratum, mori_mukai_fano_proofs} or \cite[\S12]{algebraic_geometry_5}), there is a unique smooth Fano threefold of anticanonical degree 56, namely $\PP (\cO_{\PP^2} \oplus \cO_{\PP^2}(1))$. 
\end{proof}

\subsection{Another sufficient condition for non-smoothability}
\label{sec:almost_flat_triangles}

In addition to the result of Proposition~\ref{prop:3folds_non_smoothable_isolated_sing}, here we present another obstruction for the smoothability of a toric Fano threefold with Gorenstein singularities.

\begin{theorem}[{\cite{petracci_local_to_global_obstruction}}]
\label{thm:almost_flat_triangles}
Let $N$ be a lattice of rank $3$, let $M = \Hom_\ZZ(N,\ZZ)$, let $\langle \cdot, \cdot \rangle \colon M \times N \to \ZZ$ be the duality pairing, let $P$ be a reflexive polytope in $N$, and let $X$ be the toric Fano threefold associated to the spanning fan of $P$. Assume that there are two adjacent facets $F_0$ and $F_1$ of $P$ such that:
\begin{enumerate}
\item both $F_0$ and $F_1$ are $A_n$-triangles for some integer $n \geq 1$ (see the definition in Example~\ref{ex:polygons});
\item $F_0 \cap F_1$ is a segment with $n+2$ lattice points;
\item $\langle w_1, v_0 \rangle = 0$, where $w_1 \in M$ is such that $F_1 \subseteq \{ v \in N_\RR \mid \langle w_1, v \rangle = 1 \}$ and $v_0 \in N$ is the vertex of $F_0$ which does not lie on the segment $F_0 \cap F_1$.
\end{enumerate}
Then $X$ is not smoothable.
\end{theorem}

With the terminology of \cite{petracci_local_to_global_obstruction}, the two triangles $F_0$ and $F_1$ are called ``two adjacent almost-flat $A_n$-triangles''.

\begin{proof}[Sketch of the proof of Theorem~\ref{thm:almost_flat_triangles}]
We refer the reader to \cite{petracci_local_to_global_obstruction} for all the details missing here.
Let $U_i$ be the toric open affine subscheme of $X$ associated to the facet $F_i$, for each $i = 0, 1$. Set $U = U_0 \cup U_1$.

One can show that $U$ admits an $A_n$-bundle structure over $\PP^1$. More precisely, one can construct a toric morphism $\pi \colon U \to \PP^1$ such that, for each $i=0,1$, if $V_i$ denotes the $i$th standard affine chart of $\PP^1$ then $\pi \inv (V_i) = U_i$ and the restriction $\pi \vert_{U_i} \colon U_i \to V_i$ is the projection
$
\Spec \CC[x,y,z,w]/(xy-z^{n+1}) \to \Spec \CC[w]
$.
This $A_n$ bundle may be non trivial, depending on the relative position of the two triangles $F_0$ and $F_1$.
Set $d = \langle w_1, v_0 \rangle$.
By \cite[Proposition~3.5]{petracci_local_to_global_obstruction}
there exists an isomorphism of coherent sheaves on $\PP^1$:
\[
\pi_* \cExt^1_{\cO_U} (\Omega_U, \cO_U) \simeq \bigoplus_{2 \leq j \leq n+1} \cO_{\PP^1}(-jd-j).
\]
Since $d = 0$, the sheaf on the right is a direct sum of negative line bundles on $\PP^1$, hence we have $\rH^0(U,\cExt^1_{\cO_U} (\Omega_U, \cO_U)) = 0$. By Corollary~\ref{cor:2conditions_implies_not_smoothable}, $X$ is not smoothable.
\end{proof}

With the same technique of the theorem above one can also construct some rigid toric Fano threefolds with only $cA_1$-singularities (see \cite[Theorem~1.2]{petracci_local_to_global_obstruction}).
This refutes a conjecture of Prokhorov \cite{prokhorov_degree_Fano_threefolds} according to which every Fano threefold with compound Du Val singularities is smoothable.

\subsection{Other methods}
\label{sec:other}

Here we briefly collect some other results on deformations and smoothings of toric Fano varieties. Most of these results have been motivated by Mirror Symmetry for Fano varieties (see \cite{mirror_symmetry_fano, procams, victor_lg, victor_birational_geometry, katzarkov_victor, quantum_periods_3folds, sigma,przyjalkowski_landau_ginzburg_fano}).

By analysing cluster transformations of tori, Akhtar--Coates--Galkin--Kasprzyk \cite{sigma} have introduced the notion of \emph{mutation} of Fano polytopes.
A mutation is a combinatorial procedure that, under certain conditions, transforms a Fano polytope $P$ into another Fano polytope $P'$. Ilten \cite{ilten_mutations} has proved that mutations of Fano polytopes induce deformations of the corresponding toric Fano varieties; more precisely, if $P$ and $P'$ are related via a mutation, then he has constructed a flat family over $\PP^1$ such that the fibre over $0$ is $X_P$ and the fibre over $\infty$ is $X_{P'}$.

Ilten, Lewis and Przyjalkowski \cite{ilp_toric_degenerations} have constructed toric degenerations of smooth Fano threefolds with Picard rank 1.

Christophersen and Ilten \cite{christophersen_ilten_degenerations_fano} have constructed degenerations of smooth Fano threefolds of low degree to certain unobstructed Fano Stanley-Reisner schemes. Since these unobstructed Fano Stanley-Reisner schemes are also degenerations of singular toric Fano varieties, this implies the following result.

\begin{theorem}[Christophersen--Ilten {\cite[Proposition~4.2, Theorem~5.1, Theorem~7.1]{christophersen_ilten_hilbert}}]
Let $X$ be a toric Fano threefold with Gorenstein singularities. If $(-K_X)^3 \in \{ 4, 6, 8, 10, 12 \}$, then $X$ is smoothable.
\end{theorem}

Coates--Kasprzyk--Prince \cite{laurent_inversion} have introduced a combinatorial gadget, called \emph{scaffolding}, on a Fano polytope $P$ which induces a closed embedding of the toric Fano variety $X_P$ into a bigger toric variety $Y$. Often $X_P$ is a complete intersection in the Cox coordinates of $Y$, therefore it is easy to construct embedded deformations of $X_P$ in $Y$. In many cases this produces smoothings of $X_P$. For instance, Cavey and Prince \cite{cavey_prince} have successfully applied the scaffolding method to construct deformations of toric del Pezzo surfaces to del Pezzo surfaces with a single $\frac{1}{k}(1,1)$ singularity.

Moreover, Prince \cite{prince_cracked} has found necessary and sufficient conditions in order to have that the ambient toric variety $Y$ is smooth: this is the notion of \emph{cracked polytope}. He has also found a sufficient condition for a smoothing of $X_P$ to exist inside $Y$.
Via the scaffolding method and cracked polytopes, in \cite{prince_cracked_fano} he constructs a degeneration of each smooth Fano threefold with very ample anticanonical bundle and Picard rank $\geq 2$ to a Gorenstein toric Fano threefold.


\section{Lists of reflexive polytopes of dimension 3}
\label{sec:lists}

Below we write lists of reflexive polytopes of dimension 3 which satisfy specific properties. There are exactly 4319 reflexive polytopes of dimension 3: the classification is due to Kreuzer and Skarke \cite{kreuzer_skarke_reflexive_3topes}.
The IDs we use are numbers between 1 and 4319 and come from the Graded Ring Database \cite{grdb}, which has been developed by Gavin Brown and Alexander Kasprzyk. We have produced these lists by working on the MAGMA machine on \url{fano.ma.ic.ac.uk} at the Department of Mathematics at Imperial College London; this machine contains the database of canonical polytopes of dimension 3 and is maintained by Tom Coates and Alexander Kasprzyk, both of whom we heartily thank.

\medskip

All polytopes we consider below are reflexive of dimension 3. They correspond to toric Fano threefolds with Gorenstein singularities. We denote by $X_P$ the toric Fano threefold associated to the spanning fan of $P$.

Let $\listpoly{smoothable}$ be the set of polytopes $P$ such that the corresponding toric Fano threefold $X_P$ is smoothable. It is an open question to explicitly compute $\listpoly{smoothable}$.

Let $\listpoly{smooth}$ be the set of polytopes which have only standard triangles as facets. These 18 polytopes correspond to the smooth toric Fano threefolds.

Let $\listpoly{isol}$ be the set of polytopes with unitary edges such that at least one facet is not a standard triangle. 
These 137 polytopes correspond to the singular toric Fano threefolds with isolated Gorenstein singularities.

Let $\listpoly{nodes}$ be the set of polytopes such that all facets are either standard triangles or standard squares and there is at least a square facet. These 82 polytopes correspond to the singular toric Fano threefolds with at most ordinary double points, or equivalently to the singular toric Fano threefolds with Gorenstein terminal singularities.
By Corollary~\ref{cor:fano_3fold_nodes} these varieties are smoothable.

Let $\listpoly{low}$ be the set of polytopes $P$ such that the normalised volume of the polar $P^*$ of $P$ belongs to $\{4,6,8,10,12\}$. These 220 polytopes correspond to the toric  Gorenstein Fano threefolds $X$ such that $(-K_X)^3 \in \{4,6,8,10,12\}$.

Let $\listpoly{indec}$ be the set of polytopes which contain a facet $F$ which has unitary edges, is Minkowski indecomposable and is not a standard triangle. By Proposition~\ref{prop:3folds_non_smoothable_isolated_sing} the corresponding toric Fano threefolds are not smoothable.

Let $\listpoly{aft}$ be the set of polytopes which contain a pair of adjacent almost-flat $A_n$-triangles, for some $n \geq 1$. In other words, the set $\listpoly{aft}$ contains exactly all polytopes $P$ to which Theorem~\ref{thm:almost_flat_triangles} applies. Therefore, the corresponding toric Fano threefolds are not smoothable.

Let $\listpoly{}$ denote the set of all reflexive polytopes of dimension 3, i.e.\ the set of positive integers not greater than 4319.
We have:
\begin{gather*}
\listpoly{nodes} \subseteq \listpoly{isol} \subseteq \listpoly{} \setminus \listpoly{smooth}, \\
\listpoly{smooth} \cup \listpoly{nodes} \cup \listpoly{low}  \subseteq \listpoly{smoothable}, \\
\listpoly{indec} \cup \listpoly{aft} \subseteq \listpoly{} \setminus \listpoly{smoothable}.
\end{gather*}
Below we write down the elements of most of the sets mentioned above.

\smallskip

$\listpoly{smooth}$ = \{1, 5, 6, 7, 8, 25, 26, 27, 28, 29, 30, 31, 82, 83, 84, 85, 219, 220\}

\smallskip

$\listpoly{isol}$ = \{3, 4, 11, 12, 17, 21, 22, 23, 24, 42, 48, 49, 50, 51, 54, 68, 69, 70, 71, 72,
73, 74, 75, 76, 77, 78, 79, 80, 81, 155, 156, 158, 159, 160, 167, 168, 170, 177,
187, 188, 198, 199, 200, 201, 202, 203, 204, 205, 206, 207, 208, 209, 210, 211,
212, 213, 214, 215, 216, 217, 218, 360, 363, 364, 365, 366, 376, 377, 378, 380,
385, 403, 410, 411, 412, 413, 414, 415, 416, 417, 418, 419, 420, 421, 422, 423,
424, 425, 426, 427, 686, 688, 689, 692, 693, 694, 695, 696, 707, 710, 725, 729,
730, 731, 732, 733, 734, 735, 736, 737, 738, 739, 740, 741, 1085, 1086, 1087,
1091, 1092, 1093, 1109, 1110, 1111, 1112, 1113, 1114, 1517, 1518, 1519, 1524,
1528, 1529, 1530, 1941, 1943, 2355, 2356\}

\smallskip

$\listpoly{nodes}$ = \{4, 21, 22, 23, 24, 68, 69, 70, 71, 72, 73, 74, 75, 76, 77, 78, 79, 80, 81, 198, 199, 
200, 201, 202, 203, 204, 205, 206, 207, 208, 209, 210, 211, 212, 213, 214, 215, 216, 217, 
218, 410, 411, 412, 413, 414, 415, 416, 417, 418, 419, 420, 421, 422, 423, 424, 425, 426, 
427, 729, 730, 731, 732, 733, 734, 735, 736, 737, 738, 739, 740, 741, 1109, 1110, 1111, 
1112, 1113, 1114, 1528, 1529, 1530, 1943, 2356\}

\smallskip

$\listpoly{low}$ = \{1946, 2711, 2756, 2817, 3043, 3051, 3053, 3079, 3314, 3319, 3329, 3331, 3349, 
3350, 3390, 3393, 3406, 3416, 3447, 3452, 3453, 3505, 3573, 3620, 3625, 3626, 
3667, 3683, 3702, 3727, 3728, 3731, 3733, 3735, 3736, 3738, 3739, 3740, 3756, 
3760, 3762, 3777, 3790, 3791, 3792, 3795, 3796, 3844, 3845, 3846, 3848, 3853, 
3857, 3868, 3869, 3874, 3875, 3879, 3901, 3903, 3922, 3923, 3927, 3928, 3933, 
3936, 3937, 3938, 3946, 3962, 3964, 3965, 3966, 3967, 3981, 3983, 3984, 3985, 
3991, 3995, 4003, 4004, 4005, 4006, 4007, 4022, 4023, 4024, 4027, 4031, 4032, 
4041, 4042, 4043, 4044, 4056, 4058, 4059, 4060, 4070, 4074, 4075, 4076, 4080, 
4088, 4092, 4094, 4095, 4102, 4104, 4117, 4118, 4119, 4122, 4124, 4131, 4132, 
4133, 4134, 4135, 4143, 4144, 4145, 4149, 4159, 4160, 4161, 4167, 4168, 4169, 
4170, 4179, 4180, 4181, 4182, 4183, 4184, 4186, 4190, 4191, 4194, 4200, 4202, 
4203, 4205, 4206, 4214, 4215, 4216, 4217, 4218, 4219, 4220, 4225, 4228, 4229, 
4231, 4232, 4233, 4235, 4236, 4238, 4239, 4241, 4244, 4245, 4246, 4247, 4249, 
4250, 4251, 4252, 4254, 4255, 4256, 4258, 4260, 4261, 4263, 4267, 4268, 4269, 
4270, 4272, 4273, 4275, 4278, 4280, 4281, 4282, 4284, 4285, 4286, 4287, 4288, 
4290, 4291, 4292, 4293, 4294, 4295, 4297, 4298, 4299, 4300, 4301, 4303, 4304, 
4307, 4308, 4309, 4310, 4311, 4312, 4313, 4314, 4315, 4317, 4318, 4319\}

\smallskip

$\listpoly{indec}$ = \{3, 12, 17, 32, 38, 48, 49, 51, 54, 88, 91, 94, 98, 99, 100, 101, 102, 103, 105, 115, 119, 121, 134, 137, 138, 141, 142, 155, 158, 159, 170, 188, 228, 
235, 239, 242, 243, 247, 248, 252, 254, 256, 260, 262, 265, 271, 278, 293, 294, 298, 299, 301, 317, 318, 330, 351, 353, 360, 378, 380, 438, 439, 440, 
443, 445, 455, 468, 480, 491, 492, 493, 497, 501, 502, 515, 525, 526, 529, 530, 532, 539, 541, 543, 546, 550, 553, 562, 570, 575, 604, 608, 609, 614, 
620, 645, 650, 660, 663, 688, 744, 752, 753, 754, 756, 760, 774, 775, 776, 780, 784, 790, 791, 792, 800, 834, 841, 844, 845, 852, 856, 859, 864, 866, 
887, 900, 908, 912, 914, 923, 935, 963, 979, 990, 991, 1012, 1019, 1020, 1130, 1151, 1154, 1183, 1199, 1204, 1205, 1208, 1215, 1218, 1220, 1261, 1275, 
1277, 1283, 1299, 1302, 1309, 1311, 1352, 1370, 1384, 1397, 1547, 1585, 1598, 1631, 1636, 1638, 1679, 1683, 1687, 1693, 1728, 1750, 1751, 1777, 1791, 
1992, 2014, 2046, 2047, 2050, 2051, 2080, 2081, 2084, 2096, 2124, 2129, 2379, 2404, 2425, 2427, 2455, 2456, 2716, 2750, 2751, 2755\}

\smallskip

$\listpoly{aft}$ = 
\{15, 16, 36, 41, 45, 53, 58, 59, 61, 65, 66, 102, 105, 110, 111,
112, 113, 116, 117, 124, 125, 128, 135, 141, 142, 144, 146, 147, 148, 149, 152, 162,
172, 179, 183, 189, 192, 193, 197, 230, 236, 244, 248, 261, 268, 271, 272, 277, 278,
279, 280, 281, 282, 286, 288, 290, 292, 302, 310, 324, 325, 327, 331, 332, 333, 334,
335, 337, 340, 343, 347, 349, 351, 355, 356, 358, 361, 362, 386, 399, 400, 407, 443,
445, 448, 452, 453, 456, 457, 463, 467, 487, 490, 496, 497, 499, 501, 502, 505, 507,
508, 509, 511, 512, 516, 523, 540, 545, 550, 563, 569, 577, 579, 581, 582, 583, 594,
599, 600, 601, 605, 606, 617, 629, 633, 658, 670, 671, 672, 674, 679, 682, 687, 705,
760, 764, 770, 771, 780, 781, 786, 787, 792, 797, 799, 809, 811, 812, 815, 816, 824,
859, 865, 868, 873, 875, 878, 883, 884, 889, 891, 892, 893, 894, 895, 902, 905, 929,
956, 960, 965, 987, 1003, 1004, 1006, 1011, 1021, 1038, 1045, 1051, 1156, 1160, 1168,
1175, 1177, 1199, 1203, 1209, 1216, 1217, 1225, 1232, 1234, 1251, 1252, 1253, 1255,
1256, 1260, 1262, 1265, 1275, 1286, 1287, 1293, 1300, 1305, 1308, 1324, 1327, 1351,
1371, 1383, 1398, 1533, 1545, 1550, 1551, 1554, 1561, 1579, 1589, 1613, 1614, 1615,
1620, 1637, 1638, 1656, 1665, 1666, 1671, 1686, 1690, 1693, 1697, 1711, 1747, 1748,
1760, 1763, 1989, 2000, 2001, 2027, 2045, 2051, 2052, 2068, 2071, 2072, 2076, 2084,
2096, 2098, 2102, 2379, 2380, 2385, 2403, 2405, 2423, 2424, 2425, 2427, 2738, 2777,
2778, 2792, 3047, 3057, 3063, 3064\}

\smallskip

We have $\vert \listpoly{indec} \cup \listpoly{aft} \vert = 442$. Therefore there exist at least 442 non-smoothable toric Fano threefolds with Gorenstein singularities.

\bibliography{Biblio_survey}


\end{document}